\RequirePackage{amsmath}
\RequirePackage{fix-cm}
\documentclass[smallextended]{svjour3hack}       
\usepackage{cases}
\usepackage{graphicx}
\usepackage{amssymb}
\usepackage[lofdepth,lotdepth,caption=false]{subfig}
\usepackage{placeins}
\usepackage{pdfpages}
\usepackage[hypertexnames=false,colorlinks=true,breaklinks=true,bookmarks=true,urlcolor=blue,citecolor=blue,linkcolor=blue,bookmarksopen=false,draft=false]{hyperref}

\usepackage[ruled,vlined,linesnumbered,norelsize]{algorithm2e}

\smartqed  

\spnewtheorem{thm}{Theorem}{\bf}{\rm}
\spnewtheorem{cor}[thm]{Corollary}{\bf}{\rm}
\spnewtheorem{lem}[thm]{Lemma}{\bf}{\rm}
\spnewtheorem{prop}[thm]{Proposition}{\bf}{\rm}
\spnewtheorem{defn}[thm]{Definition}{\bf}{\rm}
\spnewtheorem{rem}[thm]{Remark}{\bf}{\rm}

\DeclareMathOperator{\Diag}{Diag}
\DeclareMathOperator{\diag}{diag}
\DeclareMathOperator{\ldet}{ldet}
\DeclareMathOperator{\conv}{conv}
\DeclareMathOperator{\Trace}{Trace}

\journalname{XXX}

\allowdisplaybreaks

\begin{document}

\title{Mixing convex-optimization bounds for maximum-entropy sampling}%
\author{Zhongzhu Chen  \and \hbox{Marcia Fampa}  \and \hbox{Am\'elie Lambert} \and \hbox{Jon Lee}} %

\institute{Z. Chen \at
              University of Michigan, Ann Arbor, MI, USA \\
              \email{zhongzhc@umich.edu}
              \and
           M. Fampa \at
              Universidade Federal do Rio de Janeiro, Brazil \\
              \email{fampa@cos.ufrj.br}           
           \and
           A. Lambert \at
              Conservatoire National des Arts et M\'etiers, Paris, France \\
              \email{amelie.lambert@cnam.fr}
              \and
            J. Lee \at
              University of Michigan, Ann Arbor, MI, USA \\
              \email{jonxlee@umich.edu}        
}

\authorrunning{Chen,   Fampa, Lambert \& Lee}

\date{\today}

\maketitle

\begin{abstract}
The maximum-entropy sampling problem is a fundamental and challenging combinatorial-optimization problem,
with application in spatial statistics.
It asks to find a maximum-determinant order-$s$ principal submatrix of an order-$n$ covariance matrix.
Exact solution methods for this NP-hard problem are based on
a branch-and-bound framework. Many of the known upper bounds for the optimal value are based on
convex optimization. We present a  methodology for ``mixing'' these bounds to achieve better bounds.
\keywords{maximum-entropy sampling \and convex optimization}
\subclass{90C25 \and 90C27 \and 90C51 \and 62K99 \and 62H11}
\end{abstract}


\section*{Introduction}
Let $C$ be an order-$n$ (symmetric) positive-definite real matrix, and let $s$ be an integer
satisfying $1\leq s\leq n$. Let $N:=\{1,2,\ldots,n\}$. We interpret $C$ as the covariance matrix
for a multivariate Gaussian random vector $Y_N$.
For nonempty
$S\subseteq N$, let $C[S,S]$ denote the principle submatrix of $C$ indexed by $S$.
We denote $\log \det(\cdot)$ by $\ldet(\cdot)$. Up to constants, $\ldet C[S,S]$
is the (differential) entropy associated with the subvector $Y_S$.
The \emph{maximum-entropy sampling problem} (MESP)
is
\[
z(C,s):=\max~ \left\{  \ldet C[S,S] ~:~ |S|=s,~ S\subseteq N\right\}
\]
(see \cite{SW}).

MESP is NP-hard (see \cite{KLQ}), and the main paradigm for exact solution of moderate-sized instances
is branch-and-bound (see \cite{KLQ}). In this context,
there has been considerable work on efficiently calculating good upper bounds for MESP; see
\cite{KLQ,AFLW_IPCO,AFLW_Using,HLW,LeeWilliamsILP,AnstreicherLee_Masked,BurerLee,Anstreicher_BQP_entropy,Kurt_linx},
the survey \cite{LeeEnv}, and the closely related works \cite{LeeLind2019,AFLW_Remote}.

A very relevant point for us is the following identity:
\[
\det C[S,S] = \det C \times \det C^{-1}[\bar{S},\bar{S}],
\]
where $\bar{S}$ denotes the complement of $S$ in $\{1,2,\ldots,n\}$.
With this identity,  we have $z(C,s)=\ldet C + z(C^{-1},n-s)$, and so upper bounds for
$z(C^{-1},n-s)$  yield upper bounds for $z(C,s)$, shifting by $\ldet C$.
This idea gives something for bounds that are \emph{not} invariant under
complementation (see \cite{AFLW_Using,HLW,LeeWilliamsILP,AnstreicherLee_Masked,Anstreicher_BQP_entropy}). It does not
give us anything for bounds that are  invariant under
complementation (see \cite{KLQ,Kurt_linx}).

In \S\ref{sec:general}, we describe a very simple
general idea for ``mixing'' bounds.
In \S\ref{sec:mixingBQPwiththecomplement}, we apply the simple
idea to MESP by mixing the so-called ``BQP bound'' (see \cite{Anstreicher_BQP_entropy}) with the same bound applied to
the complementary problem.
In \S\ref{sec:mixingNLPwiththecomplement}, we mix the so-called ``NLP bound'' (see \cite{AFLW_Using})
with the same bound applied to
the complementary problem.
Because the BQP bound and the NLP bound are not invariant under complementation,
we can get improved bounds with these mixings.
In \S\ref{sec:mixingLinx}, we look at tuning the so-called ``linx bound'' (see \cite{Kurt_linx}).
In \S\ref{sec:mixingNLPand another}, we investigate mixing the NLP bound (or its complement) with
a ``non-NLP bound'' (e.g., the linx bound, the BQP bound, or the complementary BQP bound).
In \S\ref{sec:conc}, we make some concluding remarks.

Throughout, when we carry out computational experiments with the BQP bound, the complementary BQP bound,
and the linx bound,
we use SDPT3 (see \cite{SDPT3,Toh2012}) via Matlab and Yalmip (a Matlab toolbox for optimization; see \cite{Lofberg2004}).
SDPT3 has an efficient way of handling $\ldet$,
and this functionality is exposed via Yalmip  (not, at this writing, by CVX).
But when we work with the NLP bound, we employ our own tailored interior-point solver.

Further notation: We denote transpose of a vector $x$ by $x'$, and likewise
for matrices.
 $S^n(\mathbb{R})$ denotes the set of real order-$n$ symmetric matrices.
$A\circ B$ denotes Hadamard (element-wise) product of compatible matrices $A$ and $B$,
while $A\bullet B:=\Trace(AB')$ denotes the matrix dot-product.
For $X\in \mathbb{R}^{n\times n}$, $\Diag(X):=(X_{1,1},X_{2,2},\ldots,X_{n,n})'\in\mathbb{R}^n$.
For $x\in \mathbb{R}^n$, $\diag(x)\in \mathbb{R}^{n\times n}$ is defined by
$\diag(x)_{i,i}:=x_{i,i}$ and $\diag(x)_{i,j}:=0$ for $i\not=j$.

\section{General mixing}\label{sec:general}

The idea is so simple that we do not dare claim that
it is original. We are however confident that it is new
in the context of the MESP. In this section, we describe the
general idea.

We start with a combinatorial maximization problem
\[
z:=\max \{ f(S) ~:~ S\in\mathcal{F}\},
\]
where $\mathcal{F}$ is an arbitrary subset of the
power set of $\{1,2,\ldots,n\}$.
We consider $m$ upper bounds for $z$ based on convex
relaxations in a possibly lifted space of variables.

 As is standard, for $x\in\mathbb{R}^n$,
we denote the \emph{support} of $x$ by $\underbar{x}:=\{i\in \{1,2,\ldots,n\}
: x_i\not=0\}$. Also, if $x\in\{0,1\}^n$, then $x$ is the
\emph{characteristic vector} of $\underbar{x}$.

For $i=1,2,\ldots,m$, the convex set $\mathcal{P}_i$ uses variables
$(x,\mathcal{X}^i)$. The vector $x\in [0,1]^n$
relaxes $x\in \{0,1\}^n$ and is used to model
$\mathcal{F}$. Specifically, we assume that
if we project $\mathcal{P}_i$ onto $\mathbb{R}^n$,
we get a subset of $[0,1]^n$,
and then if we intersect with $\mathbb{Z}^n$,
 we get precisely the characteristic vectors
of $\mathcal{F}$.
 Next, for $i=1,2,\ldots,m$,
we have a concave function $f_i$, taking
$(x,\mathcal{X}^i)\in \mathcal{P}_i$ to $\mathbb{R}$.
We assume that for $(x,\mathcal{X}^i)\in \mathcal{P}_i$
such that $x\in\mathbb{Z}^n$, we have $f_i(x,\mathcal{X}^i)
= f(\underbar{x})$.
In this sense,
each $\mathcal{P}_i$ is an \emph{exact relaxation}
(possibly in a extended space) of $\conv\left(\{x\in\mathbb{R}^n : \underbar{x} \in \mathcal{F}\}\right)$.

Now, for $i=1,2,\ldots,m$, we have the convex programs
\[
v_i:= \max\left\{ f_i(x,\mathcal{X}^i) ~:~ (x,\mathcal{X}^i)\in \mathcal{P}_i\right\},
\]
yielding $m$ upper bounds on $z$.

Next, for $\alpha\in\mathbb{R}^m$, such that $\alpha\geq 0$, $e'\alpha=1$, we
define the \emph{mixing bound}
\[
v(\alpha) := \max\left\{ \sum_{i=1}^m \alpha_i f_i(x,\mathcal{X}^i) ~:~
(x,\mathcal{X}^i)\in \mathcal{P}_i~,~ 1=1,2,\ldots,m
\right\}.
\]

The following is very simple to establish.
\begin{prop}
The function $v(\alpha)$ is convex on $\left\{\alpha\in\mathbb{R}^m ~:~ \alpha\geq 0\right\}$, and
for all $\alpha\in\mathbb{R}^m$ such that $e'\alpha=1$, we
have $v(\alpha)\geq z$.
\end{prop}
Owing to this, a natural goal is to minimize the convex $v(\alpha)$, over
$\left\{\alpha\in\mathbb{R}^m ~:~\right.\allowbreak \left. e'\alpha=1,~ \alpha\geq 0\right\}$.
\emph{The power of the mixing bound is that the same variable $x$ is appearing in each of the $\mathcal{P}_i$~.}
If it were not for this, then the minimum value of $v(\alpha)$,
over $\left\{\alpha\in\mathbb{R}^m ~:~\right.\allowbreak \left. e'\alpha=1,~ \alpha\geq 0\right\}$, would trivially be
$\max_{i=1}^m v_i$~.

Of course each $\mathcal{P}_i$ can
be strengthened to improve the mixing bound.
But very importantly, we note that the mixing bound can be strengthened by
introducing valid equations and inequalities across the
entire variable space: $x, \mathcal{X}_1,\ldots,\mathcal{X}_m$.
We exploit both of these observations in the next section.

Before continuing, we wish to mention that a slightly different formulation for
finding an optimal mixing is as the following convex program.

\begin{align*}
&\max\ v\\
&\mbox{subject to:}\\
&v \leq f_i(x,\mathcal{X}^i),~ i=1,2,\ldots,m;\\
&(x,\mathcal{X}^i)\in \mathcal{P}_i~,~ 1=1,2,\ldots,m.
\end{align*}

The equivalence can easily be seen by Lagrangian duality.
We prefer our formulation because by aggregating the nonlinearities
into the objective, in the style of a surrogate dual, we
get a formulation that is more easily handled by solvers
and more easily optimized in terms of selecting good
mixing (and other bound) parameters.
Related to this, in the context of branch-and-bound, we can expect
that child subproblems will be able to inherit good parameters
from their parents, leading to faster computations.


\section{Mixing the BQP bound with the complementary BQP bound}\label{sec:mixingBQPwiththecomplement}

In this section, we apply the simple mixing idea from \S\ref{sec:general},
mixing the (scaled) BQP bound for MESP (see \cite{Anstreicher_BQP_entropy}) with the same bound
applied to the complementary problem. We will see that minimizing this bound over
$\alpha$ gives us a bound that is sometimes stronger than the two bounds that it is based upon
--- it is always at least as strong. In fact, we will see that the bound will tend to be stronger
when the two bounds being mixed have similar values.

\subsection{Mixing BQP and its complement}

 Let
\begin{align*}
P(n,s):=&\{
(x,X)\in\mathbb{R}^n\times S^n(\mathbb{R}) ~:~ \\
&\quad X -xx'\succeq 0,~ \Diag(X)=x,~ e'x=s,~ Xe=sx
\}\\
Q(n,n-s):&=\{
(y,Y)\in\mathbb{R}^n\times S^n(\mathbb{R})  ~:~\\
&\quad Y-yy''\succeq 0,~ \Diag(Y)=y,~ e'y=n-s,~ Ye=(n-s)y
\}.
\end{align*}
The set $P(n,s)$ (respectively, $Q(n,n-s)$) is the well-known SDP relaxation of
the binary solutions to $X-xx'=0$, $e'x=s$  (respectively, $Y-yy'=0$, $e'y=n-s$).

We introduce the \emph{mixed BQP (mBQP) bound}:
\begin{align*}
&v(C,s;\alpha,\gamma_1,\gamma_2):=\\
& \max~ (1-\alpha)\left(
 \ldet \left(
\gamma_1 C\circ X + I - {\rm diag}(x)
\right) - s {\rm log} \gamma_1
\right)\\
& \quad +
\alpha\left(
 \ldet \left(
\gamma_2 C^{-1}\circ Y + I - {\rm diag}(y)
\right) - (n-s) {\rm log} \gamma_2
+ \ldet  C
\right),\\
&\mbox{subject to:}\\
&\quad (x,X)\in P(n,s),~ (y,Y) \in Q(n,n-s),~
 x+y=e,
\end{align*}
where $0\leq \alpha \leq 1$ is a ``weighting'' parameter, and $\gamma_1,\gamma_2>0$ are
``scaling parameters''. We will see that this mBQP bound is a manifestation of the idea from \S\ref{sec:general},
mixing the scaled BQP bound with its complement.

It is almost immediate that the mBQP bound is a mixing in the precise sense of \S\ref{sec:general},
but because of
the way that we have formulated it with different variables for the complementary part,
there is a little checking to do.

We define an invertible linear map $\Phi$ by
\[
\Phi(x,X) = (e-x, X+ee'-ex'-xe').
\]
Notice that if $(\hat{y},\hat{Y)}:=\Phi(\hat{x},\hat{X})$,
then $\hat{Y}_{ij}= \hat{X}_{ij} +1 - \hat{x}_j - \hat{x}_i$.

We have the following useful result.
\begin{lem}\label{complementary_lift}
$(\hat{x},\hat{X})\in P(n,s)$ if and only if $\Phi(\hat{x},\hat{X})\in Q(n,n-s)$.
\end{lem}
\begin{proof}
We check the constraints:
\begin{align*}
&\hat{Y}-\hat{y}\hat{y}' = \hat{X} + ee' - e\hat{x}' - \hat{x}e' - (e-\hat{x})(e-\hat{x})'
= \hat{X}-\hat{x}\hat{x} \succeq 0'~.\\
&\Diag(\hat{Y})=\Diag(\hat{X}) + \Diag(ee') - \Diag(e\hat{x}') - \Diag(\hat{x}e')
=\hat{x} + e - \hat{x} -\hat{x} = \hat{y}~.\\
&e'\hat{y} = e'(e-\hat{x})=n-s~.\\
&\hat{Y}e= \!\left(\hat{X} + ee' - e\hat{x}' - \hat{x}e'\right)\! e = s\hat{x} + ne -se -n\hat{x} =(n-s)(e-\hat{x})=(n-s)\hat{y}~.
\end{align*}
The other direction is similar.
\qed
\end{proof}

For $\alpha=0$ and $\alpha=1$, the mBQP reduces to the bounds of  \cite{Anstreicher_BQP_entropy}\footnote{Helmberg
suggested (essentially) the BQP bound in 1995  (see \cite{LeeEnv,FedorovLee}) to Anstreicher and Lee, but no one
developed it at all until \cite{Anstreicher_BQP_entropy} did so extensively, drawing in and significantly extending
some techniques from \cite{AFLW_Using}.}:
\begin{prop}
$v(C,s;\alpha=0,\gamma_1,\gamma_2)$ is equal to the scaled BQP bound
\begin{align*}
- s {\rm log} \gamma_1 ~+~ &\max~ \ldet \left(
\gamma_1 C\circ X + I - {\rm diag}(x)
\right) ~,\\
&\mbox{subject to:}\\
&\quad (x,X)\in P(n,s),
\end{align*}
and   $v(C,s;\alpha=1,\gamma_1,\gamma_2)$  is equal to the scaled complementary BQP bound
\begin{align*}
 \ldet  C - (n-s) {\rm log} \gamma_2 ~+~ &\max~ \ldet \left(
\gamma_2 C^{-1}\circ Y + I - {\rm diag}(y)
\right) \\
&\mbox{subject to:}\\
&\quad (y,Y)\in Q(n,n-s).
\end{align*}
\end{prop}
\begin{proof}
When $\alpha=0$, for any $(\hat{x},\hat{X})\in P(n,s)$,
Lemma \ref{complementary_lift} allows us to always be able to choose a
$(\hat{y},\hat{Y})$, which together with $(\hat{x},\hat{X})$ is feasible
for the mBQP optimization formulation.
And because $\alpha=0$, the choice of $(\hat{y},\hat{Y})$ has no impact
on the mBQP objective function.
Similarly, when $\alpha=1$, for any  $(\hat{y},\hat{Y})\in Q(n,n-s)$,
Lemma \ref{complementary_lift} allows us to always be able to choose a
$(\hat{x},\hat{X})$
which together with $(\hat{y},\hat{Y})$ is feasible
for the mBQP optimization formulation.
And because $\alpha=1$, the choice of $(\hat{x},\hat{X})$ has no impact
on the mBQP objective function.
\qed
\end{proof}

Of course we have
\begin{prop}
\[
z(C,s)\leq  v(C,s;\alpha,\gamma_1,\gamma_2)~.
\]
\end{prop}

 We can see from the convexity of $v$
that there is a good potential
to improve on the minimum of the scaled BQP bound and the
scaled complementary BQP bound precisely when these two bounds are similar.
See Figure \ref{graph_convex} where this is illustrated the well-known ``$n=63$'' benchmark covariance matrix.
A simple univariate search
can find a good value for $\alpha$. Moreover, in the context of branch-and-bound
for exact solution of the MESP, a good (starting) value of $\alpha$ can be inherited from
a parent.

\begin{figure}[ht!]
\captionsetup[subfigure]{labelformat=empty,justification=centering}
    \centering
    \subfloat[][]{
        \includegraphics[width=0.47\textwidth]{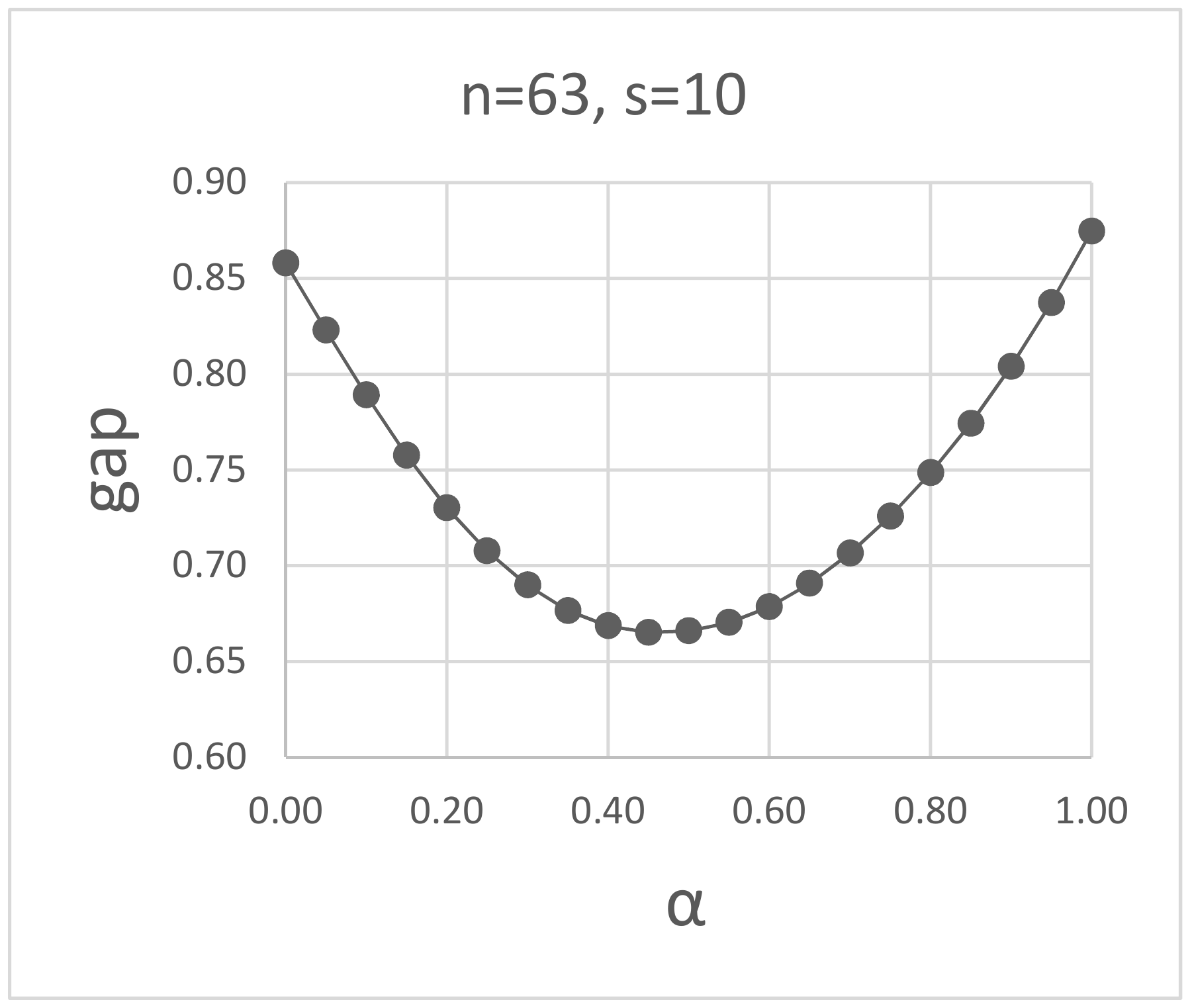}
}
   \subfloat[][]{
       \includegraphics[width=0.47\textwidth]{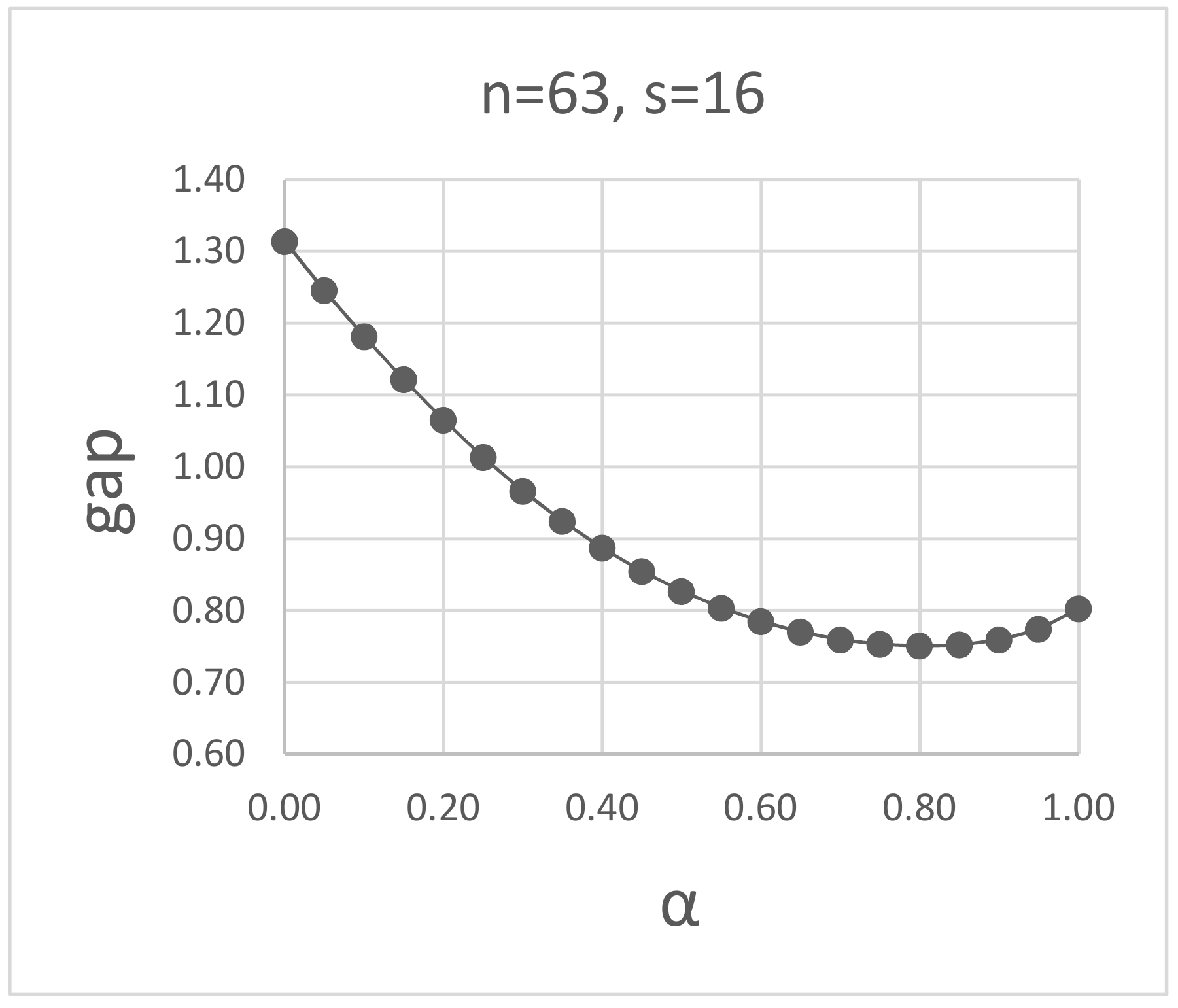}
}
    \caption{Gap vs. $\alpha$ (optimized $\gamma_i$)}\label{graph_convex}
\end{figure}


\subsection{Valid equations in the extended spaces}

Next, we will see that we can strengthen the mBQP bound,
using equations that link the extended variables from the two bounds that
we mix, and then even eliminate the variables $(y,Y)$.

\begin{prop}\label{prop:strengthenBQPm}
\begin{align*}
&v(C,s;\alpha,\gamma_1,\gamma_2) ~\geq~ \check{v}(C,s;\alpha,\gamma_1,\gamma_2):=\\
& \max~ (1-\alpha)\left(
 \ldet \left(
\gamma_1 C\circ X + I - {\rm diag}(x)
\right) - s {\rm log} \gamma_1
\right)\\
& \quad +
\alpha\left(
 \ldet \left(
\gamma_2 C^{-1}\circ (X+ee'-ex'-xe') + I - {\rm diag}(e-x)
\right)\right. \\
& \left.\qquad- (n-s) {\rm log} \gamma_2
+ \ldet  C \vphantom{C^{-1}}
\right),\\
&\mbox{subject to:}\\
&\quad (x,X)\in P(n,s).
\end{align*}
\end{prop}

The result follows from Lemma \ref{complementary_lift} and the following simple lemma.
\begin{lem}\label{complementary_valid}
For the solutions of $x+y=e$, $X=xx'$, $Y=yy'$, the equations $Y=X+ee'-ex'-xe'$ are  valid.
\end{lem}
\begin{proof}
Under $x+y=e$, we have that
\[
0 = Y-yy' = Y-(e-x)(e-x)'=Y-ee'+ex'+xe'-xx'~.
\]
Subtracting $0=X-xx'$, we obtain the desired equations.
\qed
\end{proof}

We experimented further with the ``$n=63$''  covariance matrix.
Considering now Figure \ref{graph_improvement63},
the unmixed bounds are indicated by the lines
for ``$\alpha=0$'' and ``$\alpha=1$''. We optimized the $\gamma_i$
for these bounds (see \S\ref{sec:choosingparams}). We chose an interesting range of $s$,
where the unmixed bounds transition between which is stronger
(i.e., the lines cross). The line indicated by ``$\alpha^*$''
is the optimal mixing of the BQP bound and its complement.
Note that we only optimized $v(C,s;\alpha,\gamma_1,\gamma_2)$ on $\alpha$, keeping the
optimal $\gamma_i$ from the unmixed bounds.
A (probably small) further improvement could be obtained by
iterating between optimizing on $\alpha$ and the $\gamma_i$.
The line indicated by ``$\alpha^* \mbox{ strengthened}$''
is the  optimal mixing of the BQP bound and its complement, but now with the
valid equations in the extended space.
Note that again we only optimized $\check{v}(C,s;\alpha,\gamma_1,\gamma_2)$
 on $\alpha$, keeping the optimal $\gamma_i$ from the unmixed bounds.

We can seek to improve the mBQP bound
 by adding RLT, triangle and other inequalities, valid for the BQP,
for both $(x,X)$ and  $(y,Y)$.
We could do it directly (like \cite{Anstreicher_BQP_entropy}),
but the conic-bundle  method (see \cite{Fischer}) seems more promising,
due to the large number of inequalities to be potentially exploited.
So we   dynamically include triangle inequalities via a bundle method; specifically we use the solver SDPT3 (see \cite{SDPT3}) together with the Conic Bundle Library (see \cite{CB}) for solving the associated semidefinite programs, as described in \cite{BELW17}. In the figure,  the line ``$\alpha^* \mbox{ strengthened } + \mbox{ triangles}$'' indicates the bound obtained.

\begin{figure}[ht!]
   \centering
   \includegraphics[width=\textwidth]{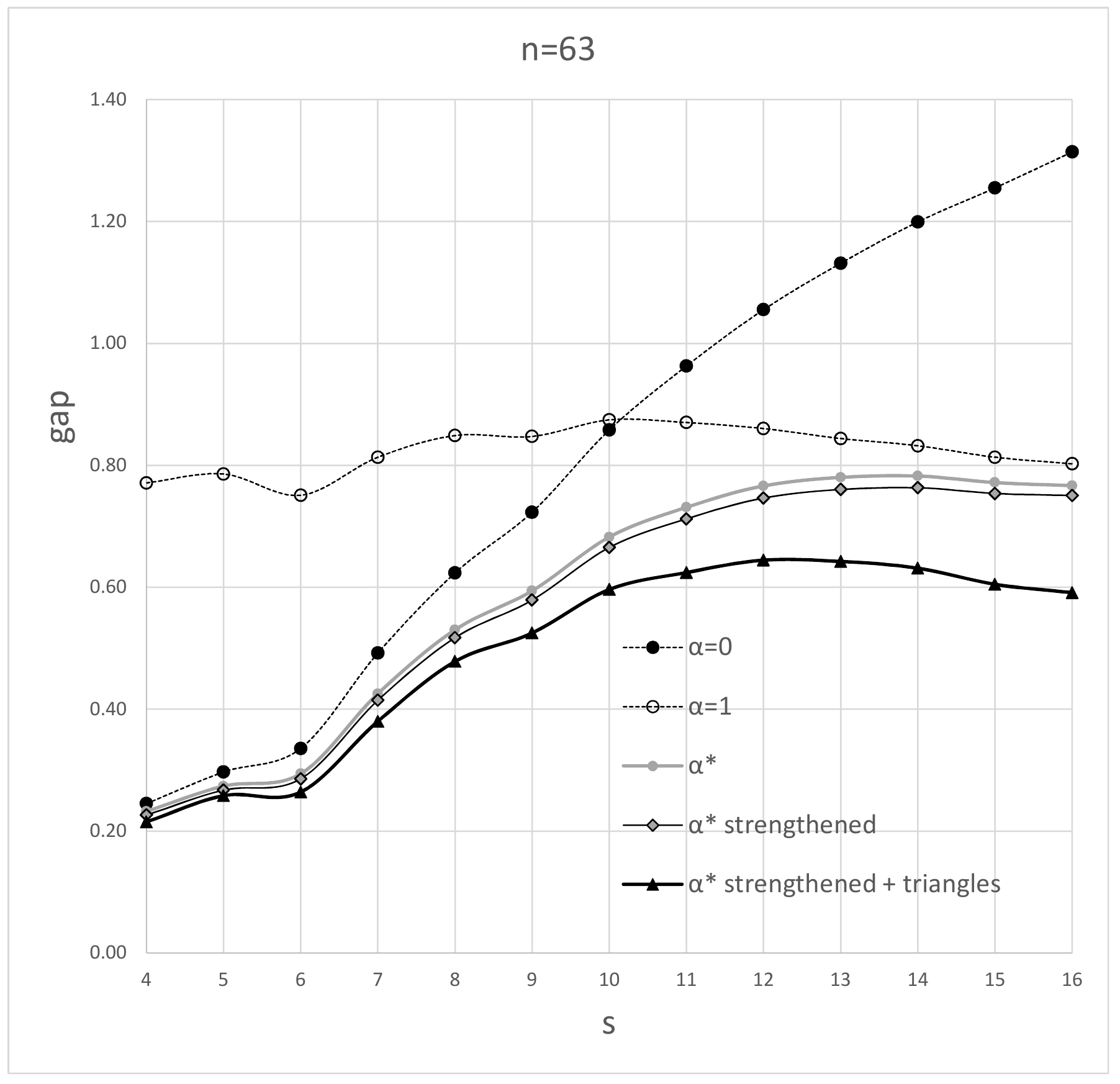}
   \caption{Gap vs. $s$ (optimized $\alpha$ and $\gamma_i$)}\label{graph_improvement63}
\end{figure}

We repeated this experiment for a the well-known larger ``$n=124$'' benchmark covariance matrix.
The results, exhibiting a similar behavior, are indicated in Figure \ref{graph_improvement124}.
Note that in this figure, gaps are to a lower bound generated by a heuristic.

\begin{figure}[ht!]
   \centering
   \includegraphics[width=\textwidth]{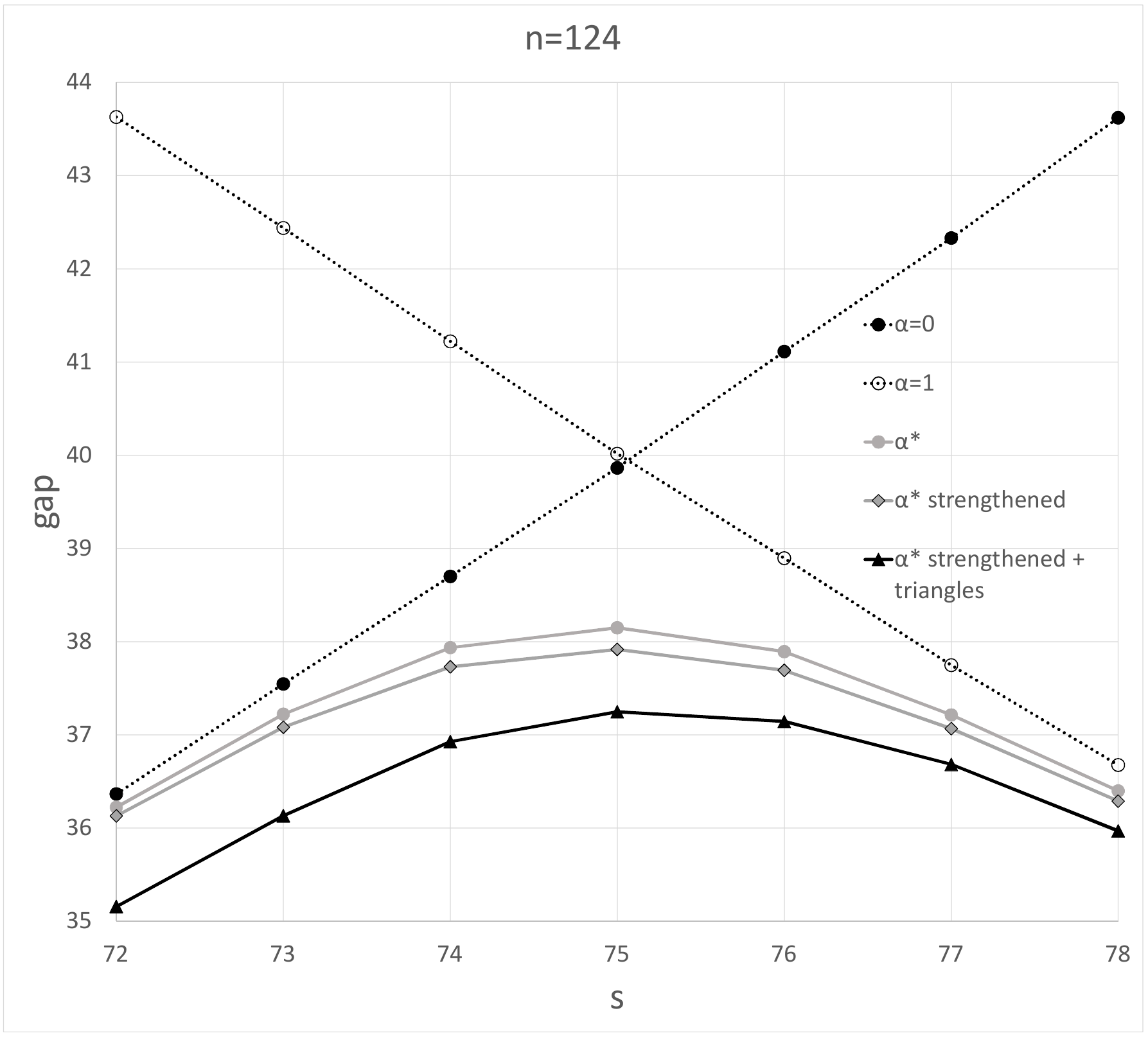}
   \caption{Gap vs. $s$ (optimized $\alpha$ and $\gamma_i$)}\label{graph_improvement124}
\end{figure}

\subsection{Choosing good parameters ($\alpha,\gamma_1,\gamma_2$)}\label{sec:choosingparams}

Toward designing a reasonable algorithm for minimizing $\check{v}(C,s;\alpha,\gamma_1,\gamma_2)$,
over $\alpha\in[0,1]$ and $\gamma_1,\gamma_2>0$,
we establish convexity properties.

\subsubsection{Convexity properties}

\begin{thm}
\label{thmconvexity}
For fixed $\gamma_1,\gamma_2>0$, the
function $\check{v}(C,s;\alpha,\gamma_1,\gamma_2)$ is convex in $\alpha\in[0,1]$.
For fixed $\alpha\in[0,1]$, the function $\check{v}(C,s;\alpha,\exp(\psi_1),\exp(\psi_2))$
is jointly convex in $(\psi_1,\psi_2)\in\mathbb{R}^2$.
\end{thm}

\begin{proof}

We already know from general principles that our mixing bounds are convex in $\alpha$.
So in this section, we begin by establishing joint convexity in the
logarithms of the scaling parameters
$\gamma_1,\gamma_2$.

Let
\begin{align}
&F_1(C,s;\gamma_1,(x,X)):= (\gamma_1C-I)\circ X +I =\gamma_1C\circ X +I - \diag(x),\label{f1}\\
&F_2(C,s;\gamma_2,(x,X)):=\gamma_2C^{-1}\circ (X + ee' - ex'-xe') + \diag(x),\label{f2}
\end{align}
\begin{align*}
&f_1(C,s;\gamma_1,(x,X)):=\ldet F_1(C,s;\gamma_1,(x,X)) - s\log \gamma_1,\\
&f_2(C,s;\gamma_2,(x,X)):=\ldet F_2(C,s;\gamma_2,(x,X)) - (n-s)\log \gamma_2 + \ldet C,
\end{align*}
\[
f(C,s;\alpha,\gamma_1,\gamma_2,(x,X)):= (1-\alpha) f_1(C,s;\gamma_1,(x,X)) + \alpha f_2(C,s;\gamma_2,(x,X)).
\]
So, with this notation,
\[
 \check{v}(C,s;\alpha,\gamma_1,\gamma_2)
= \max_{(x,X)\in P(n,s)}
(1-\alpha) f_1(C,s;\gamma_1,(x,X)) + \alpha f_2(C,s;\gamma_2,(x,X)).
\]

The function  $\check{v}(C,s;\alpha,\exp(\psi_1),\exp(\psi_2))$ is the point-wise maximum
of $f(C,s;\alpha\mathbin{,}\exp(\psi_1),\exp(\psi_2),(x,X))$, over
$(x,X)\in P(n,s)$. So it suffices to show that
$f(C,s;\alpha,\exp(\psi_1),\exp(\psi_2),(x,X))$ is itself
 convex for each fixed $(x,X)\in P(n,s)$.

In what follows, for $i=1,2$, we use $f_i$ as a short form for  $f_i(C,s;\gamma_i,  (x,X))$,
and we use $F_i(\gamma_i,(x,X))$ as a short form for $F_i(C,s;\gamma_i,(x,X))$.
We have
\begin{align*}
&\frac{\partial  f_1}{\partial\gamma_1} =\frac{\partial}{\partial\gamma_1}\left(\ldet F_1(\gamma_1,(x,X)) - s\log \gamma_1\right)\\
&\quad =\frac{\partial}{\partial\gamma_1}\left(\ldet (\gamma_1C\circ X +I - \diag(x)) - s\log \gamma_1\right)\\
&\quad =F_1(\gamma_1,(x,X)) ^{-1}\bullet (C\circ X ) -\frac{s}{\gamma_1}\\
&\quad =\frac{1}{\gamma_1}\left(F_1(\gamma_1,(x,X)) ^{-1}\bullet (\gamma_1C\circ X ) -s\right)\\
&\quad =\frac{1}{\gamma_1}\left(F_1(\gamma_1,(x,X)) ^{-1}\bullet  F_1(\gamma_1,  (x,X)) - F_1(\gamma_1, (x,X)) ^{-1}\bullet (I - \diag(x)) -s\right)\\
&\quad =\frac{1}{\gamma_1}\left(n - s-F_1(\gamma_1,(x,X)) ^{-1}\bullet (I - \diag(x)) \right)~.
\end{align*}

Letting $\psi_1:=\log\gamma_1$, by the chain rule we have
\[
\frac{\partial f_1}{\partial\gamma_1} = \frac{\partial f_1}{\partial\psi_1} \frac{d\psi_1}{d\gamma_1}= \frac{\partial f_1}{\partial\psi_1}\frac{1}{\gamma_1}~.
\]
So we have
\begin{align*}
\frac{\partial f_1}{\partial\psi_1} &=\gamma_1\frac{\partial f_1}{\partial \gamma_1} = n - s-F_1(\exp(\psi_1),(x,X)) ^{-1}\bullet (I - \diag(x)) =:g_1(\gamma_1)~.
\end{align*}

Next, we calculate
\begin{align*}
\frac{\partial^2 f_1}{\partial\gamma_1^2} &=\frac{\partial}{\partial\gamma_1}\left(\frac{1}{\gamma_1}\left(n - s-F_1(\gamma_1,(x,X)) ^{-1}\bullet (I - \diag(x)) \right)\right)\\
&= -\frac{1}{\gamma_1^2}\left(n - s-F_1(\gamma_1,(x,X)) ^{-1}\bullet (I - \diag(x)) \right) \\
&\quad + \frac{1}{\gamma_1}(e-x)'\diag(F_1(\gamma_1,(x,X)) ^{-1}(C\circ X)F_1(\gamma_1,(x,X)) ^{-1} )~.\\
\end{align*}

So we have
\begin{align*}
\gamma_1^2\frac{\partial^2 f_1}{\partial\gamma_1^2}
&= -n + s+F_1(\gamma_1,(x,X)) ^{-1}\bullet (I - \diag(x))  \\
&\quad + \gamma_1(e-x)'\diag(F_1(\gamma_1,(x,X)) ^{-1}(C\circ X)F_1(\gamma_1,(x,X)) ^{-1} ) ~.
\end{align*}

Finally, again taking $\psi_1:=\log \gamma_1$, using the chain rule we have
\begin{align*}
&\frac{\partial^2 f_1}{\partial\psi_1^2}=\frac{\partial g_1}{\partial \psi_1} =\gamma_1\frac{\partial g_1}{\partial \gamma_1} = \gamma_1\left(\frac{\partial f_1}{\partial \gamma_1} + \gamma_1\frac{\partial^2 f_1}{\partial \gamma_1^2}\right) = \gamma_1\frac{\partial f_1}{\partial \gamma_1} + \gamma_1^2\frac{\partial^2 f_1}{\partial \gamma_1^2} \nonumber\\
&\quad =  n - s-F_1(\exp(\psi_1),(x,X)) ^{-1}\bullet (I - \diag(x))\nonumber\\
&\quad   -n + s+F_1(\exp(\psi_1),(x,X)) ^{-1}\bullet (I - \diag(x))  \nonumber\\
&\quad + \exp(\psi_1)(e-x)'\diag(F_1(\exp(\psi_1),(x,X)) ^{-1}(C\circ X)F_1(\exp(\psi_1),(x,X)) ^{-1} ) \nonumber\\
&\quad =  \exp(\psi_1)(e-x)'\diag(F_1(\exp(\psi_1),(x,X)) ^{-1}(C\circ X)F_1(\exp(\psi_1),(x,X)) ^{-1} ) ~.
\end{align*}

It remains to demonstrate that this last expression is nonnegative.
We have $C\succ 0$ and $X\succeq 0$, and therefore $C\circ X\succeq 0$ (see \cite[page 175]{SchurBook}). Then, it is also clear from \eqref{f1} that  $F_1(\exp(\psi_1,(x,X))\succ 0$.
Therefore
\[
F_1(\exp(\psi_1), (x,X)) ^{-1}(C\circ X)F_1(\exp(\psi_1), (x,X)) ^{-1} \succeq 0.
\]
So we have
\[
\frac{\partial^2 f_1}{\partial\psi_1^2} \geq 0, 
\]
and we can conclude that $f_1(\exp(\psi_1), (x,X))$ is convex in $\psi_1$.

Similarly,  $f_2(\exp(\psi_2), (x,X))$ is convex in $\psi_2$.
Finally,  for fixed $\alpha$ and
$(x,X)$, $F(\alpha,\exp(\psi_1),\exp(\psi_2),(x,X))$ is
jointly convex in $\psi_1$ and $\psi_2$ because it is a weighted sum of
$f_1(\exp(\psi_1), (x,X))$ and $f_2(\exp(\psi_1), (x,X))$.
\qed
\end{proof}

\begin{rem}
By working with the $\psi_i:=\log(\gamma_i)$ and establishing convexity, we are able to
rigorously find the best values of the $\gamma_i$.
\cite{Anstreicher_BQP_entropy} does not work that way.
Working directly with the scaling parameters, $\gamma_i$,
(of course separately for the BQP bound and the complementary BQP bound),
he heuristically sought good values for the $\gamma_i$.
\end{rem}

\subsubsection{Optimizing the parameters}

The (strengthened) mBQP bound
depends on the parameters $(\alpha, \gamma_1,\gamma_2)$.
We do not have any type of full joint convexity.
But based on Theorem \ref{thmconvexity},
to find a good upper bound, we are motivated to formulate two convex problems.

First, for given $\hat{\psi}_1$ and $\hat{\psi}_2$, we consider the convex optimization problem
\begin{equation}
\label{p1a}
\min \{V_{\hat{\psi}_1,\hat{\psi}_2}(\alpha) ~:~ \alpha\in[0,1]\}~,
\end{equation}
where
\[
\begin{array}{lrl}
V_{\hat{\psi}_1,\hat{\psi}_2}(\alpha)&:=&\check{v}(C,s;\alpha,\exp(\hat{\psi}_1),\exp(\hat{\psi}_2))\\
&=&(1-\alpha) f_1(C,s;\exp(\hat{\psi}_1),(x^*,X^*)) + \alpha f_2(C,s;\exp(\hat{\psi}_2),(x^*,X^*)),
\end{array}
\]
and $(x^*,X^*)=(x^*(\alpha),X^*(\alpha))$ solves the maximization problem in Proposition  \ref{prop:strengthenBQPm} for the given $\alpha$, when $\gamma_1=\exp(\hat{\psi}_1)$, and $\gamma_2=\exp(\hat{\psi}_2)$.

Next, for $i=1,2$, we use $f_i^*(\alpha)$ as a short form for  $f_i(C,s;\exp(\hat{\psi}_i),(x^* ,X^*))$.

We  solve  (\ref{p1a})  with a primal-dual interior-point method,  considering the following  barrier problem
\begin{equation}
\label{barrierproba}
\min \left\{
V_{\hat{\psi}_1,\hat{\psi}_2}(\alpha) -\mu ( \log(\alpha)+\log(1-\alpha)) ~:~ \alpha\in(0,1) \right\},
\end{equation}
where $\mu>0$ is the barrier parameter. Let
\[
L_{\mu,\hat{\psi}_1,\hat{\psi}_2} (\alpha) := V_{\hat{\psi}_1,\hat{\psi}_2}(\alpha) -\mu ( \log(\alpha) + \log(1-\alpha))~.
\]

We motivate our algorithm, by assuming some differentiability.
The optimality conditions for the barrier problem  is obtained by differentiating  $L_{\mu,\hat{\psi}_1,\hat{\psi}_2}$ with respect to $\alpha$,
and can be written as
\[
G_{\mu,\hat{\psi}_1,\hat{\psi}_2}( \alpha) :=  \frac{\partial  V_{\hat{\psi}_1,\hat{\psi}_2}(\alpha) }{ \partial  \alpha}  - \frac{\mu}{\alpha}  + \frac{\mu}{1-\alpha}~=~0~.
\]

We aim at improving the mBQP bound by taking Newton steps to solve the  nonlinear equation above. The search direction $\delta_\alpha$,  is defined by
\[
H_{G_{\mu,\hat{\psi}_1,\hat{\psi}_2}}( \alpha)
(\delta_\alpha)= - G_{\mu,\hat{\psi}_1,\hat{\psi}_2} ( \alpha)~,
\]
where
\[
H_{G_{\mu,\hat{\psi}_1,\hat{\psi}_2}}( \alpha)=
\frac{\partial^2  V_{\hat{\psi}_1,\hat{\psi}_2}(\alpha) }{ \partial  \alpha^2}   + \frac{\mu}{\alpha^2} + \frac{\mu}{(1-\alpha)^2}~.
\]

We note that
\[
\frac{\partial  V_{\hat{\psi}_1,\hat{\psi}_2}(\alpha) }{ \partial  \alpha}  =  -f_1^*(\alpha) +
(1-\alpha)\frac{{\color{black} \partial}  f_1^*(\alpha)}{{\color{black} \partial}  \alpha} + f_2^*(\alpha) + \alpha
\frac{{\color{black} \partial}
f_2^*(\alpha)}{{\color{black} \partial}  \alpha} ~.
\]
However, we cannot analytically compute   $\partial  f_i^*(\alpha)/\partial  \alpha$, for $i=1,2$.
Indeed, we do not even know that the $f_i^*$ are differentiable.
In the implementation of the interior-point method, we consider the following approximations:
\begin{equation}
\label{gradap}
\begin{array}{rrr}
-f_1^*(\alpha) +(1-\alpha)\frac{\partial  f_1^*(\alpha)}{\partial  \alpha} &\approx& -f_1^*(\alpha)~,\\[5pt]
f_2^*(\alpha) +\alpha\frac{\partial  f_2^*(\alpha)}{\partial  \alpha} &\approx& f_2^*(\alpha)~.\\
\end{array}
\end{equation}
We then approximate the second partial derivative $\partial^2  V_{\hat{\psi}_1,\hat{\psi}_2}(\alpha)/ \partial  \alpha^2$, also considering \eqref{gradap}. At the first iteration of the interior-point method, we approximate it by $b_0=1$, and in iteration $k\geq 0$, we compute
\[
b_{k+1}=
-\frac{ \Delta  f_1^*(\alpha)^{k+1}}{\Delta  \alpha}  +  \frac{ \Delta  f_2^*(\alpha)^{k+1}}{ \Delta  \alpha}~,
\]
where, for $i=1,2$,  $\Delta f_i^*(\alpha)^{k+1}/ \Delta \alpha :=  (f_i^*(\alpha^{k+1} )-  f_i^*(\alpha^{k})/(\alpha^{k+1} -  \alpha^{k})$  is the finite-difference approximation of the first-order partial derivative  $\partial f_i^*(\alpha)/\partial \alpha$. Following what is commonly applied in a  BFGS scheme, we update the approximation of the second partial derivative at iteration $k$ only if $b_{k+1}$ is nonnegative.

We emphasize that  to compute the search direction at each iteration of the interior-point method, we need to compute $f_i^*(\alpha)$, $i=1,2$, and therefore we need the optimal solution $(x^*,X^*)=(x^*(\alpha),X^*(\alpha))$ of the (strengthened) mBQP relaxation for the current $\alpha$, when $\gamma_1=\exp(\hat{\psi}_1)$, and $\gamma_2=\exp(\hat{\psi}_2)$. The relaxation is thus solved at each iteration of the algorithm, each time for a new $\alpha$. As $\alpha$ is a real variable, the time to minimize the mBQP bound is dominated by solving mBQP relaxations, the remaining
effort for computing the bound is negligible.

In Algorithm \ref{alg:ipm_iteration}, we present, in detail, an iteration of the interior-point method. The iteration presented is repeated for a fixed value of the barrier parameter
 $\mu$, for a prescribed number of  times or until the absolute value of the residual $r$ is small enough.  The parameter  $\mu$ is then reduced and the process repeated, until $\mu$ is also small enough.

\begin{algorithm}[!ht]
\caption{Updating  $\alpha$, with $\gamma_1=\exp(\hat{\psi}_1)$, $\gamma_2=\exp(\hat{\psi}_2$) }
\label{alg:ipm_iteration}
\begin{quote}
Input:  $k$, $\alpha^{k}$, $(x^*(\alpha^k),X^*(\alpha^{k}))$,  $f_1^*(\alpha^k)$, $f_2^*(\alpha^k)$, $b_{k}$, $\mu^k$, $\tau_{\alpha}:=0.9$, $\tau_{\mu}:=0.1$.\\
Compute the residual:
\[
r:=  -f_1^*(\alpha^k) + f_2^*(\alpha^k) - \frac{\mu^k}{\alpha^k}  + \frac{\mu^k}{1-\alpha^k}.
\]
\\
Compute the search direction $\delta_\alpha$:
\[
\delta_\alpha=-r/ \left(b_k   + \frac{\mu^k}{(\alpha^k)^2} + \frac{\mu^k}{(1-\alpha^k)^2}\right).
\]
\\
Update $\alpha$:
\[
\alpha^{k+1} := \alpha^{k} + \hat{\theta} \delta_\alpha,
\]
where
\[
\textstyle \hat{\theta}:=\tau_{\alpha}\times\min\{1,\mbox{argmax}_{\theta}\{ \alpha^{k} + \theta\delta_\alpha \geq 0\}, \mbox{argmax}_{\theta}\{ \alpha^{k} + \theta\delta_\alpha \leq 1\}\}.
\]
\\
Obtain the optimal solution $(x^*(\alpha^{k+1}),X^*(\alpha^{k+1}))$ of the mBQP relaxation, considering $\alpha:=\alpha^{k+1}$,  $\gamma_1:=\exp(\hat{\psi}_1)$, and $\gamma_2:=\exp(\hat{\psi}_2)$.
\\
For $i=1,2$, set:
\begin{align*}
&f_i^*(\alpha^{k+1}):=f_i(C,s;\exp(\hat{\psi}_i),(x^*(\alpha^{k+1}), X^*(\alpha^{k+1})), \\
&\Delta_i :=    (f_i^*(\alpha^{k+1} )-  f_i^*(\alpha^{k})/(\alpha^{k+1} -  \alpha^{k}).
\end{align*}
\If {( $-\Delta_1 + \Delta_2 > 0$) } {\[b_{k+1} = -\Delta_1 + \Delta_2,\]}
\Else{ \[b_{k+1}:= b_k.\]}
Output: $\alpha^{k+1}$, $(x^*(\alpha^{k+1}),X^*(\alpha^{k+1}))$, $f_1^*(\alpha^{k+1})$, $f_2^*(\alpha^{k+1})$, $b_{k+1}$.
\end{quote}
\end{algorithm}

\bigskip

In what follows, we also define  for  given $\hat{\alpha}\in[0,1]$,  the convex problem
\begin{equation}
\label{p1psi}
\min \{ V_{\hat{\alpha}}(\psi_1,\psi_2) ~:~  (\psi_1,\psi_2)\in\mathbb{R}^2 \}~,
\end{equation}
\begin{align*}
&V_{\hat{\alpha}}(\psi_1,\psi_2):=\check{v}(C,s;\hat{\alpha}, \exp(\psi_1),\exp(\psi_2))\\
&\quad =(1-\hat{\alpha}) f_1(C,s;\exp(\psi_1),(x^*,X^*)) + \hat{\alpha} f_2(C,s;\exp(\psi_2),(x^*,X^*)),
\end{align*}
and $(x^*,X^*)=(x^*(\psi_1,\psi_2),X^*(\psi_1,\psi_2))$ solves the maximization problem in Proposition  \ref{prop:strengthenBQPm} for  $\alpha=\hat{\alpha}$,    $\gamma_1=\exp(\psi_1)$, and $\gamma_2=\exp(\psi_2)$.

Next, for $i=1,2$, we use
$f_i^*(\psi_1,\psi_2)$ as a short form for  $f_i(C,s;\exp(\psi_i)\mathbin{,}(x^*,X^*))$,
and we use $F_i$ as a short form for $F_i(C,s;\exp(\psi_i),(x^*,X^*))$.

The optimality condition for  \eqref{p1psi}  is given by
\[
\left\{\begin{array}{ll}
\frac{\partial V_{\hat{\alpha}}(\psi_1,\psi_2)}{\partial \psi_1} \,=\,   (1-\hat{\alpha}) \frac{\partial f_1^*(\psi_1,\psi_2)}{\partial \psi_1} +  \hat{\alpha}\frac{\partial f_2^*(\psi_1,\psi_2)}{\partial \psi_1} & =0~, \\[5pt]
\frac{\partial V_{\hat{\alpha}}(\psi_1,\psi_2)}{\partial \psi_2} \,=\,  (1-\hat{\alpha}) \frac{\partial f_1^*(\psi_1,\psi_2)}{\partial \psi_2} +  \hat{\alpha}\frac{\partial f_2^*(\psi_1,\psi_2)}{\partial \psi_2} & =0~.\\
\end{array} \right.
\]
Here  we should observe that we cannot analytically compute   $\partial f_1^*(\psi_1,\psi_2)/\partial  \psi_2$ nor $\partial f_2^*(\psi_1,\psi_2)/\partial  \psi_1$.
Again, we cannot even be sure that these derivatives exist.
In the implementation of the interior-point method, we consider the following approximations:
\begin{equation}
\label{aprox2}
\left\{\begin{array}{ll}
  (1-\hat{\alpha}) \frac{\partial f_1^*(\psi_1,\psi_2)}{\partial \psi_1} +  \hat{\alpha}\frac{\partial f_2^*(\psi_1,\psi_2)}{\partial \psi_1} &\approx   (1-\hat{\alpha}) \frac{\partial f_1^*(\psi_1,\psi_2)}{\partial \psi_1}~, \\[5pt]
 (1-\hat{\alpha}) \frac{\partial f_1^*(\psi_1,\psi_2)}{\partial \psi_2} +  \hat{\alpha}\frac{\partial f_2^*(\psi_1,\psi_2)}{\partial \psi_2} & \approx   \hat{\alpha}\frac{\partial f_2^*(\psi_1,\psi_2)}{\partial \psi_2}~.\\
\end{array} \right.
\end{equation}

We obtain then the following approximation for the optimality conditions for \eqref{p1psi}:
\begin{equation}
\label{defG}
G_{\hat{\alpha}}(\psi_1,\psi_2):=\left(
\begin{array}{l}
n-s-F_1^{-1}\bullet(I-\diag(x^*))\\
s-F_2^{-1}\bullet \diag(x^*)
\end{array}
\right) = \left( \begin{array}{l} 0\\0 \end{array}
\right) .
\end{equation}
We aim now at improving the  bound by taking Newton steps to solve the  nonlinear system above. The search direction is defined by
\[
\nabla G_{\hat{\alpha}}(\psi_1,\psi_2)\left(\begin{array}{c}
 \delta_{\psi_1}\\ \delta_{\psi_2}
\end{array}\right)= - G_{\hat{\alpha}}(\psi_1,\psi_2),
\]
where,
\begin{align}
&\nabla  G_{\hat{\alpha}}(\psi_1,\psi_2)=\left(
\frac{\partial  G_{ \hat{\alpha}}(\psi_1,\psi_2)}{\partial \psi_1},\,
\frac{\partial  G_{\hat{\alpha}}(\psi_1,\psi_2)}{\partial \psi_2}
\right), \label{defnG}
\\
&\frac{\partial  G_{\hat{\alpha}}(\psi_1,\psi_2)}{\partial \psi_1}=\left(\begin{array}{c}
\exp(\psi_1)(e-x^*)'\diag\left(F_1^{-1}
(C\circ X^*)F_1^{-1}\right)\\
0
\end{array}\right),\nonumber
\\
&\frac{\partial  G_{\hat{\alpha}}(\psi_1,\psi_2)}{\partial \psi_2}\! = \!
\left(\begin{array}{c}
0\\
\exp(\psi_2){x^*}'\diag\left(F_2^{-1}(
C^{-1}\circ (X^* + ee' - e{x^*}'-x^*e')
F_2^{-1}\right)
\end{array}\right)\! . \nonumber
\end{align}

In Algorithm \ref{alg:newton_iteration}, we present  an iteration of the Newton method applied to update the parameters $\psi_1$ and $\psi_2$ in the mBQP relaxation.   The iteration presented  is repeated  for a prescribed number of  times or until the absolute value of the residuals, components of $G_{\hat{\alpha}}(\psi_1,\psi_2)$, are small enough.

\begin{algorithm}[!ht]
\caption{Updating  $\psi_1$, $\psi_2$, with $\alpha=\hat{\alpha}$ }
\label{alg:newton_iteration}
\begin{quote}
Input:  $k,  (x^*(\psi_1^{k},\psi_2^{k}),X^*(\psi_1^{k},\psi_2^{k})); \psi_i^{k}, F_i^{k}, f_i^*(\psi_1^{k},\psi_2^{k}), \, i=1,2$.\\
Compute $G_{\hat{\alpha}}(\psi_1^k,\psi_2^k)$ and  $\nabla G_{\hat{\alpha}}(\psi_1^k,\psi_2^k) $ as defined in \eqref{defG} and \eqref{defnG}.
\\
Solve the linear system to obtain the search direction  $(\delta_{\psi_1}, \delta_{\psi_2})'$:
\[
\nabla G_{\hat{\alpha}}(\psi_1^k,\psi_2^k)\left(\begin{array}{c}
 \delta_{\psi_1}\\ \delta_{\psi_2}
\end{array}\right)= - G_{\hat{\alpha}}(\psi_1^k,\psi_2^k)~,
\]
\\
For $i=1,2$, update $\psi_i$:
\[
\psi_i^{k+1} := \psi_i^{k} + \delta_{\psi_i}.
\]
\\
Obtain the optimal solution $(x^*(\psi_1^{k+1},\psi_2^{k+1}),X^*(\psi_1^{k+1},\psi_2^{k+1}))$ of the mBQP relaxation, considering $\alpha:=\hat{\alpha}$,  $\gamma_1:=\exp({\psi}_1^{k+1})$, and $\gamma_2:=\exp({\psi}_2^{k+1})$.
\\
For $i=1,2$, set:
\begin{align*}
&F_i^{k+1}:=F_i(C,s;\exp({\psi}_i^{k+1}),(x^*(\psi_1^{k+1},\psi_2^{k+1}), X^*(\psi_1^{k+1},\psi_2^{k+1})),\\
&f_i^*(\psi_1^{k+1},\psi_2^{k+1}):=f_i(C,s;\exp({\psi}_i^{k+1}),(x^*(\psi_1^{k+1},\psi_2^{k+1}), X^*(\psi_1^{k+1},\psi_2^{k+1})).
\end{align*}
\\
Output:  $(x^*(\psi_1^{k+1},\psi_2^{k+1}),X^*(\psi_1^{k+1},\psi_2^{k+1})); \psi_i^{k+1},F_i^{k+1},f_i^*(\psi_1^{k+1},\psi_2^{k+1}), \, i=1,2$.
\end{quote}
\end{algorithm}

Finally, in order to obtain a good  bound, we propose an algorithmic approach where we start from given values for the parameters $\alpha$, $\psi_1$, and $\psi_2$ and alternate between solving problems
\eqref{p1a} and \eqref{p1psi}, applying respectively, the procedures described in Algorithms \ref{alg:ipm_iteration} and \ref{alg:newton_iteration}.

\begin{figure}[!ht]
\captionsetup[subfigure]{labelformat=empty}
    \centering
\subfloat[][]{
        \includegraphics[width=0.8\textwidth]{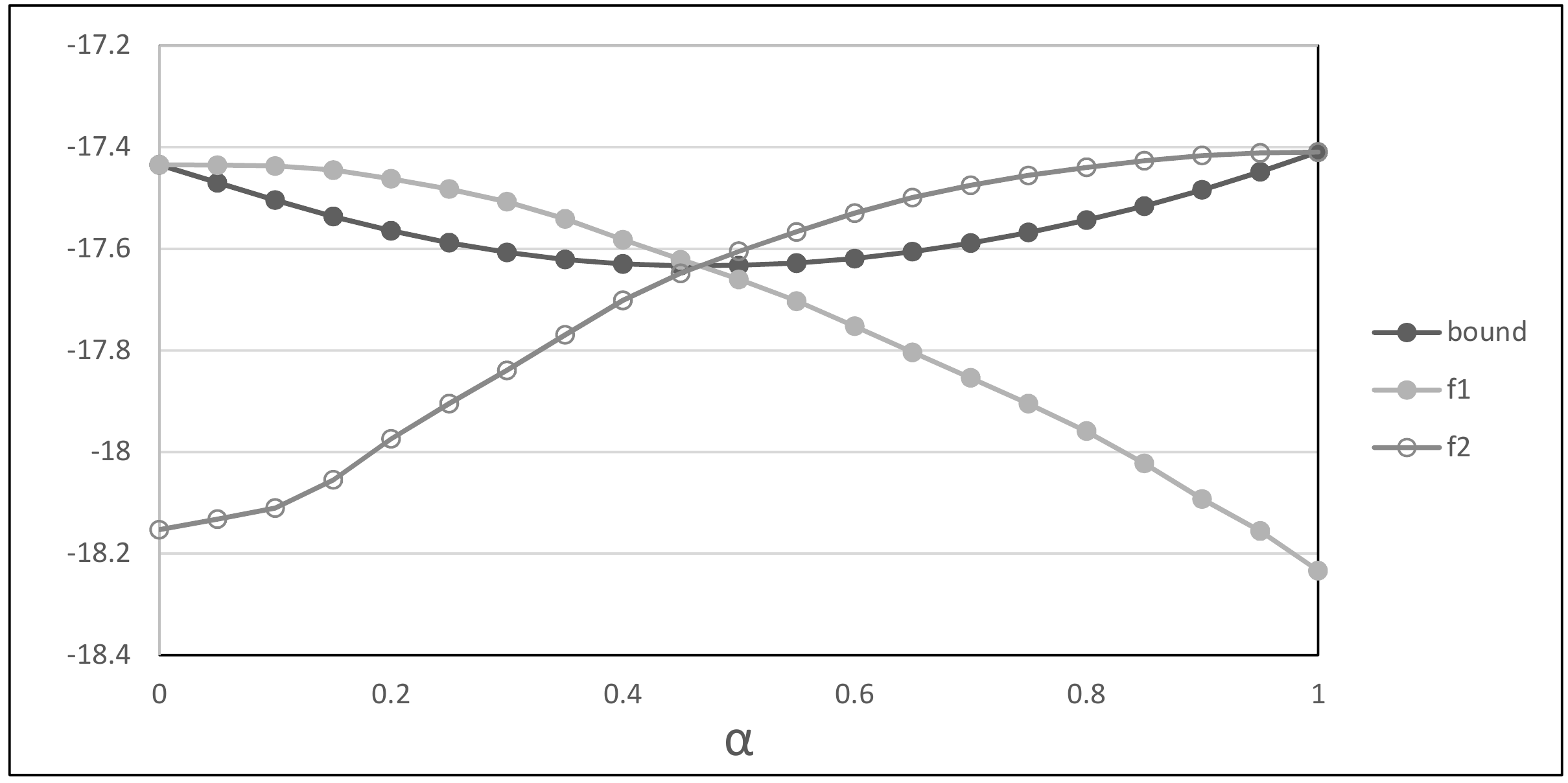}
}

\subfloat[][]{
        \includegraphics[width=0.8\textwidth]{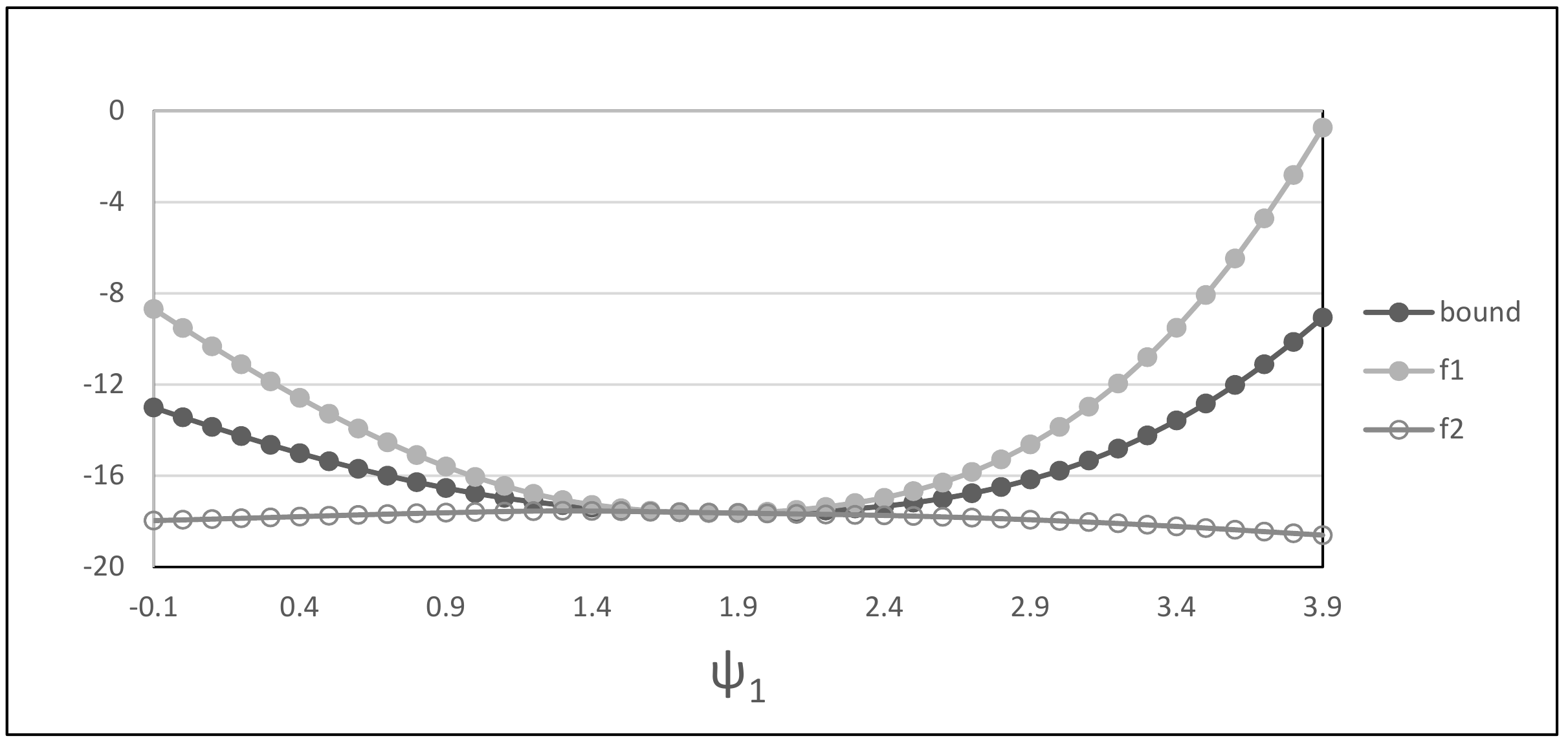}
}

\subfloat[][]{
        \includegraphics[width=0.8\textwidth]{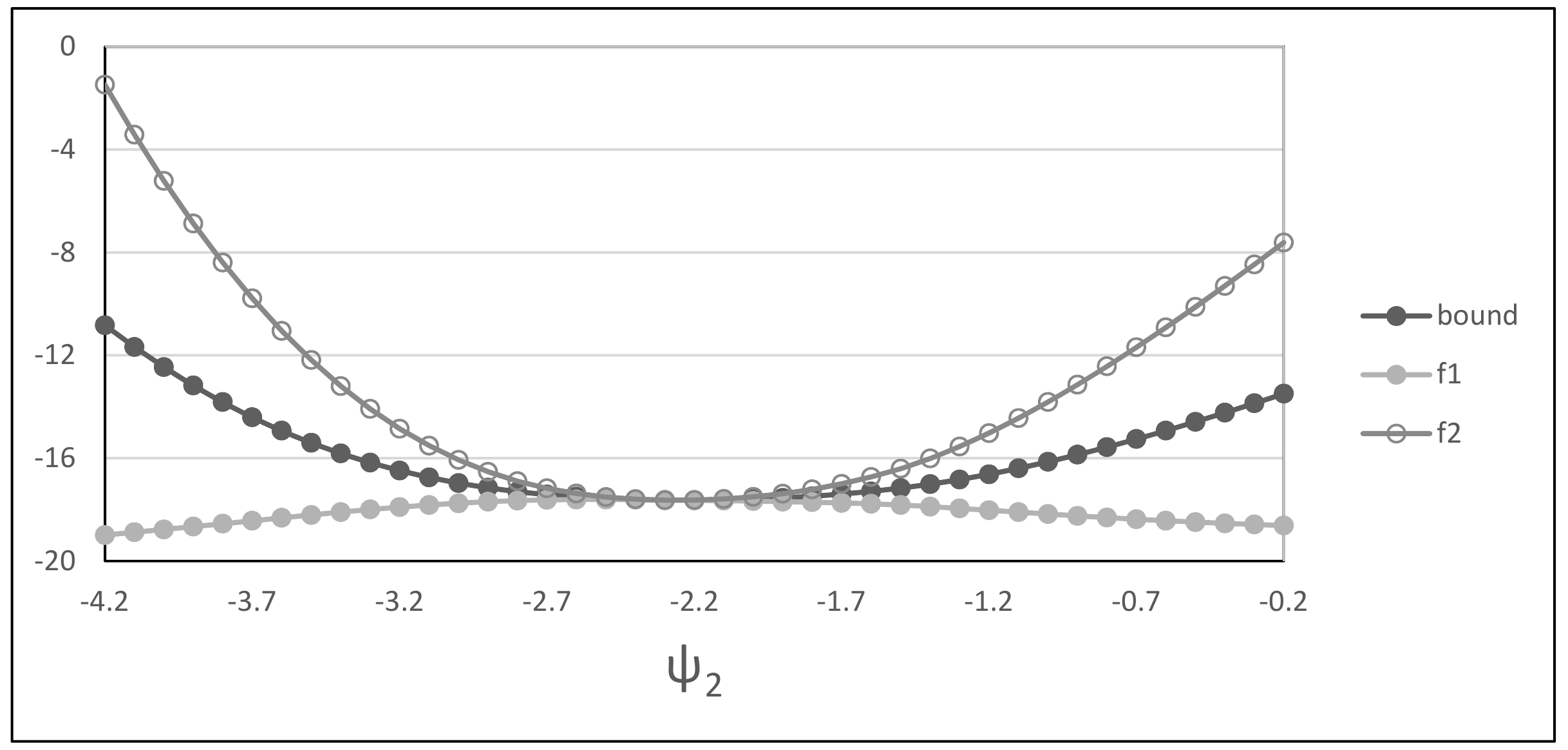}
}
\caption{Variation of $f_1,f_2$, and the overall bound, with $\alpha,\psi_1,\psi_2(n=63,s=10)$}\label{varyv1v2}
\end{figure}

%

In Figure \ref{varyv1v2} we illustrate how $f_1$, $f_2$, and the (strengthened) mBQP bound vary with each of the parameters $\alpha$, $\psi_1$, and $\psi_2$, separately,  for the instance with $n=63$, $s=10$. To construct each plot in Figure \ref{varyv1v2}, we fix two of the parameters and vary the other. The values of the two parameters that are fixed were obtained by the procedure described above, i.e., alternating between the execution of Algorithms \ref{alg:ipm_iteration} and \ref{alg:newton_iteration}. The interval in which the third parameter varies is centered in the value also obtained with the alternating algorithm, so the best bound obtained by the algorithm is depicted in the figure.
The plots in Figure \ref{varyv1v2} were considered to support  the approximations pointed in \eqref{gradap} and \eqref{aprox2}, used in the computation of the search directions of Algorithms \ref{alg:ipm_iteration} and \ref{alg:newton_iteration}.


\section{Mixing the NLP bound with the complementary NLP bound}\label{sec:mixingNLPwiththecomplement}

Now, we introduce the \emph{mixed NLP (mNLP) bound}:
\begin{align*}
&w(C,s;\alpha,\gamma_1,\gamma_2):=\\
& \max~ (1-\alpha)\left(
 \ldet \left(
\gamma_1 X^{p/2}( C-D) X^{p/2} +(\gamma_1D)^x
\right) - s {\rm log} \gamma_1
\right)\\
& \quad + \alpha \left( \ldet \left( \gamma_2 Y^{\bar{p}/2}( C^{-1}-\bar{D}) Y^{\bar{p}/2} +(\gamma_2\bar{D})^y \right)
 - (n-s) {\rm log} \gamma_2
+ \ldet  C
\right),\\
&\mbox{subject to:}\\
&e'x=s,~
 x+y=e,
\end{align*}
where $0\leq \alpha \leq 1$ is a weighting parameter.

The objective function of the mNLP relaxation is defined over the order-$n$ diagonal matrices $D$ and $\bar{D}$, the order-$n$ vectors $p$ and $\bar{p}$, and the scaling parameters $\gamma_1,\gamma_2>0$. The following notation is also employed in its definition: $X:=\diag(x)$,  $Y:=\diag(y)$, and  $(V^u)_{i,i}:= V_{i,i}^{u_i}$, $i=1,\ldots,n$, for a diagonal matrix $V$ and a vector $u$.

In \cite{AFLW_Using}, three different strategies are presented for choosing $D$, $p$, and $\gamma_1$, in order to have the NLP relaxation proven convex. Analogously, the strategies also applies to the selection of the parameters $\bar{D}$, $\bar{p}$, and $\gamma_2$, for the complementary problem. In our numerical experiments with the NLP bound, we have chosen these parameters based on the so-called ``NLP-Trace'' strategy, where $D$  minimizes the trace of $D-C$, subject to $D-C$ being positive semidefinite. Once $D$ is chosen, the scaling parameter $\gamma_1$ should be selected in the interval $[1/d_{\max}, 1/d_{\min}]$ (see \cite{AFLW_Using}). In our experiments, we have tested 100 values for $\gamma_1$ in this interval an report results for the best one. The same strategy is applied to the complementary problem. We note that the optimal scaling factors  for the mBQP bound were obtained with Newton steps in the previous section,  as described in Algorithm \ref{alg:newton_iteration}. The same methodology could not be applied here, because  the objective function of the mNLP relaxation is neither convex in the scaling parameters nor in the
logarithms of the scaling parameters. Therefore, for the results we present on the mNLP bound, we choose $\gamma_1$ to be the best scaling parameter for the original NLP bound ($\alpha=0$), among the 100 values tested, we choose
$\gamma_2$ to be the best scaling parameter for the complementary NLP bound ($\alpha=1$), among the 100 values tested. To select $\alpha$ for each instance, we obtained the mNLP bound for all $\alpha = 0.1i$, $i=0,1,\ldots,10$. The results reported correspond to the best such $\alpha$.


Finally, we note that unlike the mBQP bound,  the mNLP bound  cannot be computed by SDPT3, via Matlab and Yalmip. So, to compute it, we have coded an
interior-point algorithm, also in Matlab. The solution procedure  is the same as described in \cite[Section 3]{AFLW_Using},
where the NLP bound and the complementary NLP bound  are considered. Later, the procedure was also applied in the related
work \cite{AFLW_Remote}. The procedure employs a long-step path following methodology, using
logarithmic barrier terms for the bound constraints on $x$ (i.e., $0\leq x \leq e$).
For a fixed value of the barrier parameter $\mu$, the barrier function is approximately minimized on
$\{x\in\mathbb{R}^n : e'x=s\}$.
The parameter  $\mu$ is then reduced and the process is repeated, until $\mu$ is small enough for
an approximate minimizer to be within a prescribed tolerance of optimality. The
tolerance  is certified by a dual solution generated by the algorithm, providing a
valid upper bound for the optimal value of NLP.

In Figure \ref{fig:nlp_compnlp}, we illustrate our approach.
By mixing the NLP-Trace bound and the complementary NLP-Trace bound, we were able to
obtain an improvement for the $n=124$ problem in the vicinity of $s=73$.

\begin{figure}[!ht]
    \centering
        \includegraphics[width=0.9\textwidth]{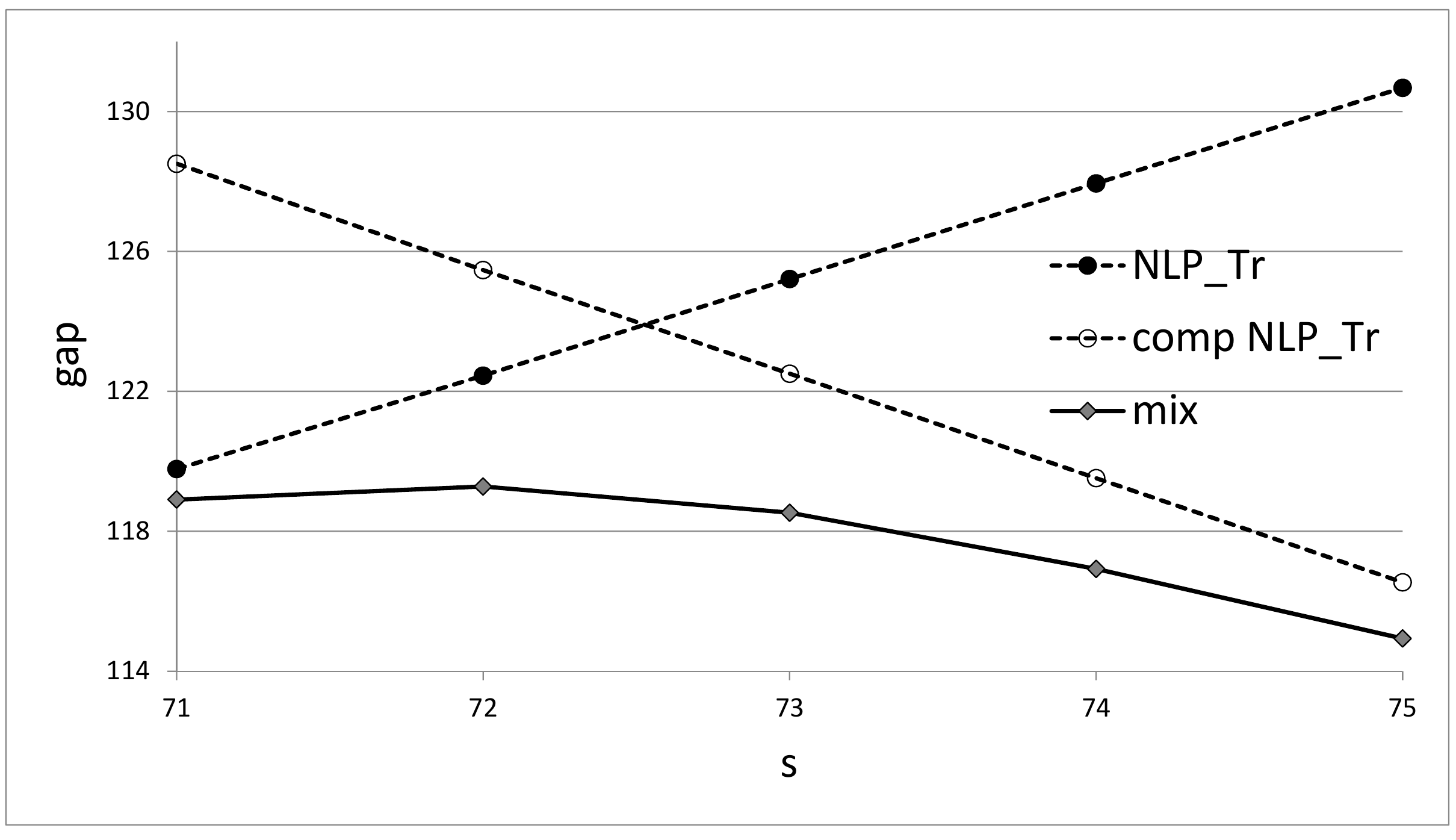}
\caption{Mixing the NLP bound with complementary NLP bound}\label{fig:nlp_compnlp}
\end{figure}



\section{On the linx bound}\label{sec:mixingLinx}

Next, we consider the linx bound introduced by \cite{Kurt_linx}, i.e., the solution of
\begin{equation}
\label{linx_bound}
\max \{ \frac{1}{2} v(\gamma,x) \, | \, e'x=s, \, 0\leq x\leq e\}
\end{equation}
where
\[
v(\gamma,x):=  \ldet F(\gamma,x) - s \log \gamma
\]
and
\begin{equation}
\label{flinx}
F(\gamma,x):= \gamma C\diag(x)C+I-\diag(x),
\end{equation}

The linx bound has excellent performance, and it is a challenge to improve upon it.
In the the remainder of this section, we consider fine tuning the bound via its scaling parameter.
In \S\ref{sec:mixingNLPand another}, we are able to get an improvement on the
linx bound by mixing it with the NLP bound.

\subsection{Optimizing the linx bound on the scaling parameter $\gamma$}

The linx bound
depends on the scaling parameter $\gamma$. \cite{Kurt_linx}
observed that the linx bound is particularly sensitive to
the choice of $\gamma$. This is probably due to
the fact that the bound is derived by bounding the \emph{square} of
the determinant of an order-$s$ principle submatrix of $C$.
So for mixing with the linx bound, it is very useful
to be able to optimize on $\gamma$.

To find the best  bound, we now define $\psi:=\log(\gamma)$ and formulate the problem
\begin{equation}
\label{p1alinx}
\min_{\psi} \, \{H(\psi)  \}~,
\end{equation}
where
\begin{equation*}
\label{Fpsislinx}
H(\psi) := v(\exp(\psi), x^*)~,
\end{equation*}
and where $x^*$  is a maximizer of \eqref{linx_bound}, with $\gamma(=\exp(\psi))$ fixed.

\begin{thm}
The function $H(\psi)$ is convex in $\psi \in \mathbb{R}$.
\end{thm}
\begin{proof}
Based on the same argument used in the proof of Theorem \ref{thmconvexity}, we show that $v(\exp(\psi),x)$ is convex in $\psi$, for fixed $x$ in the feasible set of \eqref{linx_bound}.

We have
\[
\begin{array}{lll}
\frac{\partial }{\partial \gamma} v(\gamma,x)&=&  F(\gamma,  x)^{-1}\bullet ( C\diag(x)C ) - \frac{s}{\gamma}\\
&=&  \frac{1}{\gamma} (F(\gamma,  x)^{-1}\bullet ( ( F(\gamma,x) - I +\diag(x)  ) - s)\\
&=&  \frac{1}{\gamma} (F(\gamma,  x)^{-1}\bullet ( \diag(x) - I  ) + n - s)\\
\frac{\partial^2 }{\partial \gamma^2} v(\gamma,x)&=&  \frac{\partial}{\partial \gamma} \left(  \frac{1}{\gamma} (F(\gamma,x)^{-1}\bullet ( \diag(x) - I  ) + n - s)  \right) \\
&=& - \frac{1}{\gamma^2} (F(\gamma,  x)^{-1}\bullet ( \diag(x) - I  ) + n - s) \\
&+& \frac{1}{\gamma} (e-x)'\diag(F(\gamma,x)^{-1}( C\diag(x)C)F(\gamma,x)^{-1}).
\end{array}
\]
Therefore
\[
\begin{array}{lll}
\frac{\partial }{\partial \psi} v(\gamma,x)&=&  \gamma\frac{\partial }{\partial \gamma} v(\gamma,x)\\
&=& F(\gamma,  x)^{-1}\bullet ( \diag(x) - I ) + n - s,
\end{array}
\]
and
\begin{equation}
\label{secdervlinx}
\begin{array}{lll}
\frac{\partial^2 }{\partial \psi^2} v(\gamma,x)&=&  \gamma\frac{\partial }{\partial \gamma} v(\gamma,x)  + \gamma^2\frac{\partial^2 }{\partial \gamma^2} v(\gamma,x)\\
&=& \gamma (e-x)'\diag(F(\gamma,x)^{-1}( C\diag(x)C)F(\gamma,x)^{-1}).
\end{array}
\end{equation}
Now, it remains to show that
\[
\frac{\partial^2 v}{\partial\psi^2}(\exp(\psi), x^*) \geq 0, \; \forall \psi.
\]
Considering \eqref{secdervlinx}, it suffices to show that \[\diag(F(\exp(\psi), x^*) ^{-1}( C\diag(x^*)C)F(\exp(\psi), x^*) ^{-1} ) \geq 0.\]
We have
 $C\succ 0$ and $\diag(x^*)\succeq 0$, therefore $ C\diag(x^*)C\succeq 0$. Then, it is also clear from \eqref{flinx} that  $F(\exp(\psi),x^*)\succ 0$ and, therefore,
\[F(\exp(\psi), x^*) ^{-1}(C\diag(x^*)C)F(\exp(\psi), x^*) ^{-1} \succeq 0,\]
which completes the proof.
\qed
\end{proof}

\begin{rem}
By working with  $\psi:=\log(\gamma)$ and establishing convexity, we are able to
rigorously find the best values of the $\gamma_i$.
\cite{Anstreicher_BQP_entropy} does not work that way.
Working directly with the scaling parameters, $\gamma$,
he heuristically sought a good value for $\gamma$.
\end{rem}

\subsection{The Newton method  on the variable $\psi:=\log(\gamma)$}

 The optimality condition for \eqref{p1alinx} can be written as
\[
G(\psi):=\frac{\partial }{\partial \psi} v(\exp(\psi),x^*) = n-s-F(\exp(\psi),  x^*)^{-1}\bullet(I-\diag(x^*))=0.
\]
We aim at improving the linx bound by taking Newton steps to solve the  nonlinear equation above. The Newton direction $\delta_{\psi}$ is then defined by
\[
H_G(\psi)
\delta_{\psi} = - G(\psi)~,
\]
where
\[
\begin{array}{rl}
& H_G(\psi):= \frac{\partial^2 }{\partial \psi^2} v(\exp(\psi),x^*) \\
&\quad = \exp(\psi)(e-x^*)'\diag\left(F(\exp(\psi),  x^*)^{-1}(C\diag(x^*)C)F(\exp(\psi),  x^*)^{-1}\right).
\end{array}
\]

\section{Mixing the NLP bound and a ``non-NLP bound''}\label{sec:mixingNLPand another}

A convenient solver for calculating the BQP bound and its complement
and also for calculating the linx bound is SDPT3 via Yalmip.
But the NLP bound and its complement are  not amenable to solution by SDPT3 via Yalmip.
So we developed our own IPM for calculating the NLP bound.
Because of this dichotomy between available solvers, we need a special approach for
mixing the NLP bound or its complement, with any of the BQP bound, its complement, or the linx bound.

We are not very concerned with efficiency. Rather, we only seek a practical method for
calculating these mixed bounds to see if we can get an improvement
on the unmixed bounds by mixing.

Our idea is simply to apply Lagrangian relaxation to the mixing bound, in its form
with duplicated variables, as follows:
\begin{align*}
v(\alpha) :=& \max\left\{ \alpha f_1(x,\mathcal{X})
+ (1-\alpha) f_2(y,\mathcal{Y}) \, : \,
(x,\mathcal{X})\in \mathcal{P},\, (y,\mathcal{Y})\in \mathcal{Q},\, x+y=e
\right\}\\
=& \min_{\pi\in\mathbb{R}^n} \Biggl\{\max \left\{
 \alpha f_1(x,\mathcal{X})
+ (1-\alpha) f_2(y,\mathcal{Y})
+ \pi'\left( e-x-y \right)
\right.
~:~\\
&\qquad \qquad \qquad\left.
(x,\mathcal{X})\in \mathcal{P},~ (y,\mathcal{Y})\in \mathcal{Q}
\right\}\Biggr\}.\\
=& \min_{\pi\in\mathbb{R}^n} \Biggl\{ \pi'e +
\max \left\{
 \alpha f_1(x,\mathcal{X}) - \pi'x
~:~
(x,\mathcal{X})\in \mathcal{P}
\right\}\\
&\qquad \qquad \qquad+ \max \left\{(1-\alpha) f_2(y,\mathcal{Y}) -\pi'y
~:~ (y,\mathcal{Y})\in \mathcal{Q}
\right\}
\Biggr\}.
\end{align*}
In this form, we apply subgradient optimization to find an optimal $\pi\in\mathbb{R}^n$,
and at each step the Lagrangian subproblem decouples into the $(x,\mathcal{X})\in \mathcal{P}$
maximization problem and the $(y,\mathcal{Y})\in \mathcal{Q}$  maximization problem.
So we can apply separate solvers to each.

In Figure \ref{fig:subgradmix}, we illustrate some successes with our approach.
By mixing the NLP-Trace bound and linx bound, we were able to
obtain an improvement for the $n=63$ problem in the vicinity of $s=25$.

%
%

\begin{figure}[!ht]
\captionsetup[subfigure]{labelformat=empty}
\centering
\subfloat[][]{
\includegraphics[width=.95\textwidth]{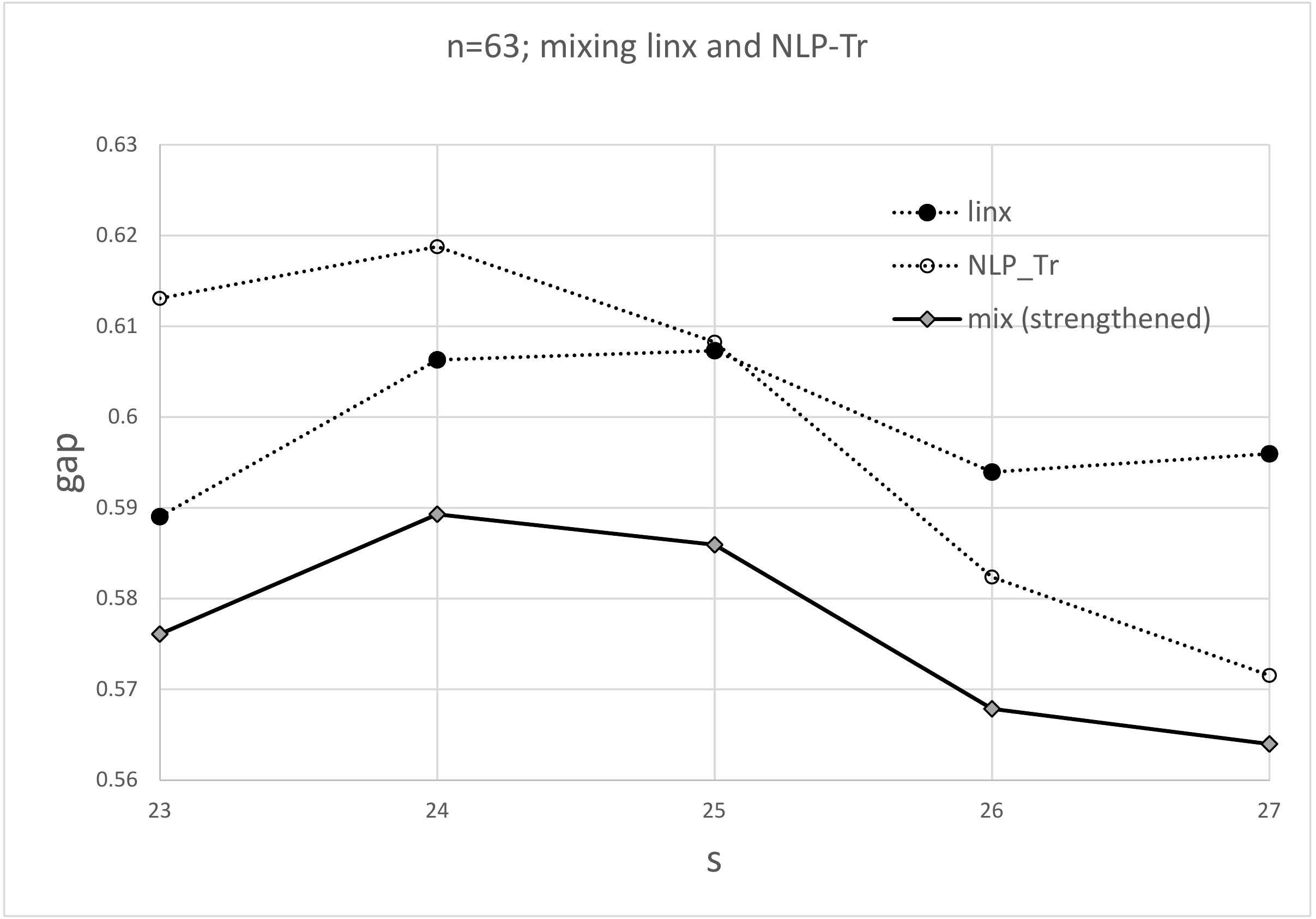}
}
\caption{Mixing the NLP bound with ``Non-NLP bounds''}\label{fig:subgradmix}
\end{figure}


\section{Concluding remarks}\label{sec:conc}

It is a challenge to \emph{efficiently} employ our ideas in the
context of branch-and-bound. We need to find effective mixing parameters
$\alpha$ quickly.
Note that in our notation, \cite{Anstreicher_BQP_entropy} is
using only $\alpha=0$ or $1$, and in the context of branch-and-bound,
each child inherits $\alpha$ from its parent,
only updating the choice occasionally.
In the context of branch-and-bound, we would now
expect that for many subproblems, we would have  $\alpha=0$ or $1$.
But we can further expect that for many we will have $0<\alpha<1$, and
we would then gain from our approach. The
guidance of \cite{Anstreicher_BQP_entropy} is:
``we use a simple criterion based on the number of fixed
variable and depth in the tree to decide when to check the
other bound''. So we would proceed similarly, doing a
univariate search for a good $\alpha$ after an inherited
value becomes stale.

It is not clear at all how our mixing idea could be adapted to
``spectral and masked spectral bounds'' (see \cite{KLQ,AnstreicherLee_Masked,BurerLee,HLW,LeeWilliamsILP}),
because these are apparently not based on convex relaxation.
We would like to highlight this as an interesting area to explore.

Very recently, \cite{Weijun} presented new results on a relaxation and
on an approximation algorithm for MESP. It will be interesting to see if some
of those results can be exploited in our context.

Finally, our general mixing idea, although well suited for MESP, should
find application on other combinatorial-optimization problems
with nonlinearities. It is a challenge to find other good applications.


\section*{Acknowledgments}
J. Lee was supported in part by ONR grant N00014-17-1-2296, AFOSR grant FA9550-19-1-0175,
and Conservatoire National des Arts et M\'etiers.
M. Fampa was supported in part by CNPq grants 303898/2016-0 and
434683/2018-3.
J. Lee and M. Fampa were supported in part by funding from the Simons
Foundation and the Centre de Recherches Math\'ematiques, through the Simons-CRM
scholar-in-residence program.
The authors thank Kurt Anstreicher for
supplying them with his (and Jon's) Matlab codes and advice in running them.
Moreover, Kurt's talk at the
2019 Oberwolfach workshop on Mixed-Integer Nonlinear Programming
catalyzed the renewed interest of Fampa and Lee in the topic of maximum-entropy sampling.


\FloatBarrier

\bibliographystyle{alpha}

\bibliography{JLee}
\end{document}